\documentclass[11pt]{amsart}
\usepackage{amsmath,amssymb,enumerate}
\usepackage{latexsym}
\usepackage{amsfonts}
\usepackage {amsbsy}
\usepackage {amsmath}
\usepackage{amssymb}
\usepackage{enumerate}
\usepackage{ bbm}
\newtheorem{theorem}{Theorem}
\newtheorem{lemma}[theorem]{Lemma}
\newtheorem{corollary}[theorem]{Corollary}
\newtheorem{proposition}[theorem]{Proposition}

\theoremstyle{definition}
\newtheorem{definition}[theorem]{Definition}
\newtheorem{example}[theorem]{Example}
\newtheorem{remark}[theorem]{Remark}

\numberwithin{theorem}{section}
\numberwithin{equation}{section}

\theoremstyle{definition}

 \numberwithin{equation}{section} 
  \numberwithin{figure}{section} 
 \theoremstyle{plain}    
 \theoremstyle{plain}    
 \theoremstyle{remark}
 \theoremstyle{remark}
 \theoremstyle{definition}
\theoremstyle{plain}  
\theoremstyle{plain}

\theoremstyle{definition}

\numberwithin{equation}{section}

\newcommand{\N}{\mathbb{N}}
\newcommand{\R}{\mathbb{R}}
\newcommand{\C}{\mathbb{C}}

\usepackage{hyperref}
\hypersetup{
    bookmarks=true,         
    unicode=false,         
    pdftoolbar=true,       
    pdfmenubar=true,       
    pdffitwindow=false,    
    pdfstartview={FitH},   
    colorlinks=true,       
   linkcolor=black,         
    citecolor=black,        
    filecolor=black}

\makeatletter
\@namedef{subjclassname@2010}{%
  \textup{2010} Mathematics Subject Classification}
\makeatother

\author{Papa Badiane}
\address{Laboratoire de Mathématiques et Applications;  Université Assane Seck de Ziguinchor BP 523.}
\email{p.badiane4963@zig.univ.sn}

\author{Ahmed Zeriahi}
\address{Institut de Math\'ematiques de Toulouse, 
Universit\'e de Toulouse, 
CNRS, UPS, 118 route de Narbonne, 
31062 Toulouse cedex 09, France}
\email{ahmed.zeriahi@math.univ-toulouse.fr}

\begin{document}


  

\title[The eigenvalue  problem]{A variational approach to   the eigenvalue problem 
for complex Hessian operators}

\date{\today}

\begin{abstract} Let $1 \leq m \leq n$ be two integers and $\Omega \Subset \C^n$ a bounded $m$-hyperconvex domain in $\C^n$. Using a variational approach, we prove the existence of the first eigenvalue and an associated eigenfunction which is $m$-subharmonic with finite energy for general twisted complex  Hessian operators of order $m$. Under some extra assumption on the twist measure we prove H\"older continuity of the corresponding eigenfunction. Moreover we give applications to the solvability of more general degenerate complex Hessian equations with the right hand side depending on the unknown function.
\end{abstract}

\maketitle

\tableofcontents

\section{Introduction}

Let $\Omega \Subset \C^n$ be bounded domain (non-empty open connected set) in $\C^n$ and $\mu \geq 0$ a positive Borel measure on $\Omega$ with positive finite mass $0 < \mu (\Omega) < + \infty$ and $m$ an integer such that $1\leq m \leq n$.

The eigenvalue problem for the twisted $m$-Hessian operator associated to $\mu$ is to find  a couple $(\lambda,u)$, where  $\lambda > 0$ is a constant and  $u \in \mathcal {SH}_m (\Omega)\cap L^{\infty}(\Omega),$ satisfying the following properties:
\begin{equation}\label{eq:VP-Hessian}
\left\{\begin{array}{lcl} 
 (dd^c u)^m \wedge \beta^{n - m} = (- \lambda \, u)^m \mu  &\hbox{on}\  \Omega, \, \, \,  \, \, (\dag)\\
  u = 0 & \,\,  \hbox{in}\  \partial \Omega,  \, \, \, \, \, \, (\dag\dag)\ \\
 u < 0\, \, \, \, & \hbox{in}\  \Omega.
\end{array}\right.
\end{equation}
The equation $(\dag)$ should be understood in the sense of currents (or Radon measures) on $\Omega$ (see Section 2) and the identity $(\dag\dag)$ means for any $\zeta \in \partial \Omega$, $\lim_{z \to \zeta} u (z) = 0$.

\smallskip

When $m=1$ and $\mu$ is the Lebesgue measure this is the classical eigenvalue problem for the Laplace operator. When $m=n$ and $\mu$ is a smooth positive volume form on $\bar \Omega$, this is the eigenvalue problem for the complex Monge-Ampère operator which we studied in our recent paper  \cite{BaZe23}.

\smallskip

Let us recall some basic  notations and definitions in order to explain the general case $1 \leq m\leq n$ and state our main results.

Recall the usual operators $d = \partial + \bar{\partial}$ and $d^c := i ( \bar{\partial} - \partial)$ so  that
$dd^c =  2 i  \partial  \bar{\partial}$.  
Given  a real function $u \in \mathcal{C}^2 (\Omega)$,  for each integer $1 \leq k \leq n$, we denote by $\sigma_k (u)$ the continuous function defined at each point $z \in  \Omega$ as the $k$-th symmetric polynomial of the eigenvalues $\lambda_1 (z), \cdots,\lambda_n(z)$ 
of the complex Hessian matrix $  \left(\frac{\partial^2 u }{\partial z_j \partial \bar{z}_k} (z)\right)$ of $u$ i.e. 
\begin{equation}  \label{eq:Symetricpolynomial}
\sigma_k (u) (z) := \sum_{1 \leq j_1 < \cdots < j_k \leq n} \lambda_{j_1} (z) \cdots \lambda_{j_k} (z), \, \,  \, \, z \in \Omega.
\end{equation}
A simple computation shows that  
\begin{equation}  \label{eq:Symetricpolynomial}
(dd^c u)^k \wedge \beta^{n - k} = \frac{(n-k)! \, k!}{n!}  \, \sigma_k (u)  \,  \beta^n, \, \, 
\end{equation}
pointwise on $\Omega$ for $1 \leq k \leq m$, where $\beta := dd^c \vert z\vert^2$ is the usual K\"ahler form on $\C^n$ (up to a constant).

We say that a real function $u \in \mathcal{C}^2 (\Omega)$ is $m$-subharmonic on $\Omega$ if for any $1 \leq k \leq m$, we have $\sigma_k (u) \geq 0$ pointwise  in $\Omega$. In particular such a function is subharmonic on $\Omega$.

Observe that  the function $u$ is $1$-subharmonic  on $\Omega$ ($m= 1$) if and only if it is  subharmonic on $\Omega$ and $\sigma_1 (u) = (1 \slash 4) \Delta u$, where $\Delta$ is the Laplace operator associated to the standard K\"ahler metric on $\C^n$,  while $u$ is  $n$-subharmonic  on $\Omega$ ($m = n$) if and only if $u$ is  plurisubharmonic on $\Omega$ and $\sigma_n (u)  = \mathrm{det}  \left(\frac{\partial^2 u }{\partial z_j \partial \bar{z}_k}\right)$.

It was shown by Z. B\l ocki  in \cite{Bl05}, that it is possible to extend the notion of $m$-subharmonic function to non smooth functions using the concept of $m$-positive currents (see Section 2). 

We denote by  $\mathcal {SH}_m (\Omega)$ the positive cone of $m$-subharmonic functions on $\Omega$.
Moreover, identifying positive $(n,n)$-currents with positive Radon measures, it is possible to define the $k$-Hessian measure $(dd^c u)^k \wedge \beta^{n - k}$ when $1 \leq k \leq m$ for any (locally) bounded $m$-subharmonic function $u$ on $\Omega$ (see Section 2). 
 
 We will use a variational approach to solve  the problem \eqref{eq:VP-Hessian}.  We define two functionals on the convex positive cone  $\mathcal E_m^1(\Omega)$ of $m$-subharmonic  functions in $\Omega$ with finite energy (see Section 2 and Section 3).
 
 \smallskip
 
 The first one is the energy functional defined on $\mathcal E_m^1(\Omega)$ by the formula:
 \begin{equation}
 E_m (\phi) = E_{m,\Omega} (\phi) := \frac{1}{m + 1} \int_\Omega (- \phi) (dd^c \phi)^m \wedge \beta^{n-m}, \, \, \,  \phi \in \mathcal{E}^1_m (\Omega).
\end{equation}  
This functional (up to the $-$ sign) is a primitive of the Hessian operator on $\mathcal{E}_m^1 (\Omega)$ i.e. $E_m'(\phi) = -  (dd^c \phi)^m \wedge \beta^{n - m}$ on $ \mathcal{E}_m^1(\Omega)$ in the sense of  Lemma \ref{lem:FirstVariation}. We will see that it is  convex (on affine lines) on $\mathcal{E}_m^1 (\Omega)$  (see Lemma \ref{lem:FirstVariation}). 

\smallskip

The second functional is attached to  a positive Borel measure $\mu$ on $\Omega$ satisfying the following integrability condition: 
\begin{equation}  \label{eq:Integrability}
\mathcal{E}^1_m (\Omega) \subset L^{m+1}(\Omega,\mu). 
\end{equation}
The functional associated to $\mu$ is defined by the following formula
\begin{equation}
  I_{\Omega,\mu,m}  (\phi) := \frac{1}{m +1}\int_\Omega (-\phi)^{m +1} d \mu, \, \, \,  \phi \in \mathcal{E}^1_m (\Omega).
\end{equation}
 This is again a convex functional on $\mathcal{E}^1_m (\Omega) $ such that $I'_{m} (\phi) = -  (-\phi)^m\mu$ for any $\phi \in  \mathcal{E}^1_m (\Omega)$ in the sense of Lemma \ref{lem:FirstVariation} below. 
 
 This shows that the equation $(dd^c \phi)^m \wedge \beta^{n - m} = (- \lambda \phi)^m \mu$ is the Euler-Lagrange equation of the  functional 
 \begin{equation} \label{eq:Functional}
    \Phi_{\Omega,\mu,m} (\phi) := E_m (\phi) -  \lambda^m I_{\Omega,\mu,m} (\phi),  \, \, \,  \phi \in \mathcal{E}^1_m (\Omega).
 \end{equation}
 When $(\Omega,\mu)$ is fixed, we will write for short $I_m =  I_{\Omega,\mu,m} $ and $\Phi_m =  \Phi_{\Omega,\mu,m}$.
 
 To  state our main results, we make the following assumptions.
 
 \smallskip
 \smallskip

 {\bf Assumptions (H)}
 \begin{itemize}
  \item $\Omega \Subset \C^n$ is $m$-hyperconvex i.e. it admits a continuous negative $m$-subharmonic exhaustion (see Definition \ref{def:hyperconvex});
  \item $\mu$ is a positive Borel measure on $\Omega$ such that $0 < \mu (\Omega) < \infty$ which is strongly diffuse with respect to the $m$-Hessian capacity  (see Definition \ref{def:Diffuse}).
\end{itemize}  
 An important example  is when  $\mu := g \beta^n$, where $0 \leq g \in L^p (\Omega)$ with $p > n\slash m$ and $\int_\Omega g \beta^n > 0$ (see Section 2 for more examples).
 
 \smallskip
 \smallskip
 
 Here is our first main result.
\begin{theorem} \label{thm:Variational1} Let $\Omega \Subset \C^n$ be a bounded domain and $\mu$ a positive Borel measure on $\Omega$ such that $(\Omega,\mu)$ satisfies the assumptions $(H)$. 

Then the following  properties hold :
\smallskip 

\noindent $(i)$ the  formula
  \begin{equation} \label{eqR1}
  \lambda_1^m := \inf \left\{\frac{E_m(\phi)}{I_m (\phi)} ; \phi \in \mathcal E^1_m(\Omega), \phi \not \equiv 0 \right\},
 \end{equation} 
 defines a positive real number $\lambda_1 = \lambda_1(m,\mu,\Omega) > 0$; 
\smallskip 

\noindent 	$(ii)$  there exists  a function $u_1 \in \mathcal E^1_m (\Omega)$  such that $u_1 \not \equiv 0$ in $\Omega$ and 
 \begin{equation} \label{eqR2}
 \lambda_1^m = \frac{E_m(u_1)}{I_m (u_1)};
 \end{equation} 
 \smallskip 
\noindent  $(iii)$   $(\lambda_1,u_1)$ is a solution to the eigenvalue equation \eqref{eq:VP-Hessian} $(\dag)$ i.e.
$$
(dd^c u_1)^m \wedge \beta^{n-m} = (-\lambda_1 u_1)^m \mu,
$$ in the sense of measures on $\Omega$.
	
\smallskip
 \noindent  $(iv)$ Moreover if  $\Omega$ is strictly $m$-pseudoconvex, $\mu = g \beta^n$ with $g \in L^p(\Omega)$ and $(m,p)$ satisfy the following conditions
 \begin{equation} \label{eq:Holder}
 (n - 1) \slash 2 < m \leq n  \, \, \, \text{ and} \, \, \, \, p > p^*(m,n),
\end{equation}
where $p^*(m,n)$ is given by the formula \eqref{eq:Conditionp*},  then  $u_1 \in C^\alpha(\bar \Omega)$ for some $\alpha \in ]0,1[$ and  $(\lambda_1,u_1)$ is a solution to the eigenvalue problem \eqref{eq:VP-Hessian}.
\end{theorem}
 The last statement $(iv)$ for $m=n$ is implicitly contained in \cite{BaZe23}.

\smallskip
\smallskip

Our second main result is the following.
\begin{theorem} Let $\Omega \Subset \C^n$ be a bounded domain and $\mu$ a positive Borel measure on $\Omega$ such that $(\Omega,\mu)$ satisfies the assumptions $(H)$.

Let $ f : \Omega \times ]-\infty, 0] \longrightarrow \R^+$ be a Borel function satisfying the following properties:

\begin{itemize}
\item for any $t \leq 0$, $f(\cdot,t) \in L^{\infty} (\Omega,\mu) $.

\item for $\mu$-a.e.  $z \in \Omega$, the function $t \mapsto f(z,t) $ is differentiable  on $]-\infty, 0[$ and 
there exists a constant $ \lambda_0 < \lambda_1$ such that 
 $$
 \partial_t f(z,t) \geq - \lambda_0, \, \, \, \text{for any} \, \, \, t < 0.
 $$
 \end{itemize}
 Then the Hessian equation 
 $$
  (dd^c \varphi)^m \wedge \beta^{n - m} = f(\cdot,\varphi)^m \mu, \, \, \,   \hbox{on}\  \Omega,
 $$
 admits a solution $\varphi \in \mathcal E^1_m(\Omega)$.
 
 Moreover if  $\Omega$ is strictly $m$-pseudoconvex, $\mu = g \beta^n$ with $g \in L^p(\Omega)$ and $(m,p)$ satisfy the condition \eqref{eq:Holder}, 
 then the following Dirichlet problem   
 \begin{equation}\label{eq:VP-Hessian2}
\left\{\begin{array}{lcl} 
   (dd^c \varphi)^m \wedge \beta^{n - m} = f(\cdot,\varphi)^m \mu, \, \, \,   \hbox{on}\  \Omega, \\
   \varphi = 0 & \, \hbox{in}\  \partial \Omega, 
   \end{array}\right.
\end{equation}
admits a  solution $ \varphi  \in \mathcal{SH}_m(\Omega) \cap C^{\alpha} (\bar \Omega)$ for some $\alpha \in ]0,1[$. 
 \end{theorem}
 
Actually we will prove a more general result (see Theorem \ref{thm:Variational2}) and give some other applications in Section 5. 
 
\smallskip

In the real case  similar results were obtained  by K. Tso  for the real Monge-Ampère operator on a bounded strictly convex domain in $\R^n$ and the Lebesgue measure on $\R^n$ (see \cite{T90}). For the real Hessian operators the existence of the first eigenvalue was proved by X.J. Wang (see \cite{W94}). These authors used a different method based on a parabolic approach.

\section{Preliminaries}

\subsection{The $m$-subharmonic functions}

Let $\Omega \Subset \C^n$ be bounded domain in $\C^n$. 
Recall the usual operators $d = \partial + \bar{\partial}$ and $d^c := i ( \bar{\partial} - \partial)$ so  that
$dd^c =  2 i  \partial  \bar{\partial}$.  

As recalled in the introduction,   a real function $u \in \mathcal{C}^2 (\Omega)$ is $m$-subharmonic on $\Omega$ if for any $1 \leq k \leq m$, we have $\sigma_k (u) \geq 0$ pointwise  on $\Omega$ (see formula \eqref{eq:Symetricpolynomial}). In particular such a function is subharmonic on $\Omega$.

Let us recall the general definition of $m$-subharmonic functions following Z. B\l ocki  in \cite{Bl05}.

Let $\C^n_{(1,1)}$ be the space of real $(1,1)$-forms on $\C^n$ with constant coefficients, and define the cone of $m$-positive $(1,1)$-forms by 
$$
\Theta_m:=\{\omega \in \C^n_{(1,1)}:~ \omega \wedge\beta^{n-1}, \cdots, \omega^m \wedge \beta^{n-m} \geq 0  \}.
$$ 

A smooth $(1,1)$-form $\omega$ on $\Omega$ is said to be $m$-positive on $\Omega$ if for any $z \in \Omega$, $\omega(z) \in \Theta_m.$

 \begin{definition}
 	    A function $u: \Omega \to \R \cup \{-\infty\}$ is said to be $m$-subharmonic on $\Omega$ if it is subharmonic on $\Omega$ $($not identically $-\infty$ in $\Omega$ $)$ and the current $dd^c u$ is $m$-positive on $\Omega$ i.e. for any collection of smooth $m$-positive $(1,1)$-forms $\omega_1, \cdots, \omega_{m-1}$ on $\Omega$, the following inequality 
  $$
 dd^cu \wedge \omega_1\wedge \cdots \wedge \omega_{m-1} \wedge \beta^{n-m} \geq 0,
 $$
 holds in the sense of currents on $\Omega.$
   \end{definition}

  We denote by $\mathcal{SH}_m(\Omega)$ the positive convex cone of $m$-subharmonic functions which are not identically $-\infty$ in $\Omega.$ These are the $m$-Hessian potentials.
  
   \subsection{Complex Hessian operators}

  Following \cite{Bl05}, we can define the Hessian operators acting on (locally) bounded $m$-subharmonic functions as follows. Given $u_1, \cdots, u_k \in \mathcal{SH}_m(\Omega) \cap L^\infty(\Omega)$ ($1\leq k \leq m$), one can define inductively the following positive ($k,k$)-current on $\Omega$
  
  $$
  dd^cu_1\wedge \cdots \wedge dd^cu_k\wedge \beta^{n-k}:=dd^c(u_1d^cu_2\wedge \cdots \wedge dd^cu_k \wedge \beta^{n-k}).
  $$
  
  In particular, if $u \in \mathcal{SH}_m(\Omega) \cap L^\infty_{loc}(\Omega)$, the positive $(m,m)$-current  
  \newline $(dd^cu)^m\wedge \beta^{n-m}$ can be identified to a positive Borel measure on $\Omega$, called $m$-Hessian measure of $ u$ defined by:
  $$
  (dd^cu)^m\wedge \beta^{n-m}= \frac{m ! (n-m)!}{n !} \sigma_m(u)\beta^n.
  $$
 
  It is then possible to extend Bedford-Taylor theory to this context. In particular, Chern-Levine Nirenberg inequalities hold and the Hessian operators are continuous under local uniform convergence and monotone convergence pointwise a.e. on 
 $\Omega$ of sequences of functions in $\mathcal{SH}_m(\Omega) \cap L^\infty_{loc}(\Omega)$ (see \cite{Bl05,Lu15}).

\smallskip
\smallskip

We will need the following definitions.
  \begin{definition} \label{def:hyperconvex}
1.	We say that the domain $\Omega \Subset \C^n$ is $m$-hyperconvex if there exists  a bounded continuous $m$-subharmonic  function $\rho: \Omega \to ]-\infty,0[$ which is exhaustive i.e. for any $c < 0$, $\{z \in \Omega \, ; \, \rho(z) < c\}\Subset \Omega$. 

\smallskip

2. We say that the domain $\Omega \Subset \C^n$ is strictly $m$-pseudoconvex if $\Omega$ admits a smooth defining function $\rho$ which is strictly $m$-subharmonic in a neighborhood of $\bar \Omega$ and satisfies $\vert \nabla \rho \vert >0$ pointwise on $\partial \Omega=\{\rho=0\}$. 
In this case we can choose $\rho$ so that
\begin{equation}
(dd^c\rho)^k \wedge \beta^{n-k} \geq \beta^n \, \, \textnormal{for} \, 1 \leq k \leq m,
\end{equation}
pointwise on $\Omega$. 
When $m=n$ we say that $\Omega $ is strictly pseudoconvex.
\end{definition}
 
  \begin{definition} 
 	Let $\Omega \Subset \C^n$ be a $m$-hyperconvex domain. The $m$-Hessian capacity is defined as follows : for any compact set $K \subset \Omega$,
  \begin{equation} \label{eq:Cap}
  c_m(K)=c_{m,\Omega}(K):=\sup\Big\{ \int_K(dd^cu)^m\wedge \beta^{n-m} \,: \, u \in \mathcal{SH}_m(\Omega), \,-1\leq u \leq 0  \Big\}.
  \end{equation}
  \end{definition}

 By Chern-Levine-Nirenberg type inequalities, this number is finite (see \cite{Lu15}).

  \subsection{The class of $m$-subharmonic functions of finite energy} Let us define some convex classes of singular plurisubharmonic functions in $\Omega$ suitable for the variational approach.
 
 \smallskip
 
  Following Cegrell \cite{Ceg98}, we define $\mathcal{E}^0_m(\Omega)$ to be the positive convex cone of negative functions $\phi \in \mathcal{SH}_m(\Omega) \cap L^\infty(\Omega)$ such that
  $$
 \int_{\Omega}(dd^c\phi)^m\wedge \beta^{n-m} <+\infty, \, \phi_{|\partial \Omega} \equiv 0. 
 $$
 Then we define $\mathcal{E}^1_m(\Omega)$ as the set of $m$-subharmonic functions $u$ in $\Omega$ such that there exists a decreasing sequence $(u_j)_{j\in \N}$ in the class $\mathcal{E}^0_m(\Omega)$ satisfying $u=\lim_j u_j$ in $\Omega$ and $\sup_j \int_{\Omega}(-u_j)(dd^cu_j)^m\wedge \beta^{n-m}<+\infty.$ 

 For $\phi \in \mathcal{E}^1_m(\Omega)$, we define its $m$-energy by 
 \begin{equation}
	E_m(\phi) = E_{m,\Omega} (\phi) :=\frac{1}{m+1}\int_{\Omega}(-\phi)(dd^c\phi)^m\wedge \beta^{n-m}.
 \end{equation}
 
 For each constant $C>0$, we define the convex set
 \begin{equation}\label{eq:Omega-C}
 \mathcal{E}_m^1(\Omega, C):=\{ \phi \in \mathcal{E}_m^1(\Omega):\,  E_m(\phi) \leq C \}.
 \end{equation}

Another concept which will be useful is defined for any Borel function $h$  locally upper bounded  on $\Omega$ by the formula
$$
P_m (h) = P_{m,\Omega} (h) := \Big(\sup \{v \in \mathcal{SH}_m(\Omega) \, ; \, v \leq h \, \, \text{in} \, \,  \Omega\}\Big)^*,
$$
where $*$ means the upper- semi-continuous regularization.

 If there exists $v_0 \in \mathcal{SH}_m(\Omega)$ such that $v_0 \leq h$ in $\Omega$, then $P_{m,\Omega} (h) \in \mathcal{SH}_m(\Omega)$ and $v_0 \leq P_{m,\Omega} (h)$ in $\Omega$. 
 
 Moreover if   $h \leq 0$ and  there exists $v_0 \in   \mathcal{E}_m^1(\Omega)$ such that $v_0 \leq h$ in $\Omega$ then $P_{m,\Omega} (h) \in \mathcal{E}_m^1(\Omega)$.
 The function $P_m (h)$ is called the $m$-subharmonic envelope of $h$ in $\Omega$.
 
 This construction is classical in  Convex Analysis as well as in Potential theory. It was introduced earlier in Pluripotential theory by J. Siciak (see \cite{Si81} and the references therein) and considered in this context in \cite{Lu15}.

 The operator $P_m$ plays a fundamental role in the variational approach. This was highlighted in   \cite{BBGZ13}  for the complex Monge-Amp\`ere equation on compact Kähler manifolds. It has been used in the case of the Monge-Amp\`ere equations in hyperconvex domains in $\C^n$ in \cite{ACC12}  and extended to the case of Hessian equations in \cite{Lu15}.

   An important particular case of this construction has been studied in this context in \cite{Lu15} (see also \cite{SaAb13}). 

Let $K \subset \Omega$ be a compact set. The relative extremal $m$-subharmonic function of the condenser $(K,\Omega)$ is defined by the formula
\begin{eqnarray} \label{eq:extremalfonction}
h_K = h_{K, \Omega} &:= & P_m (- {\bf 1}_K)  \nonumber \\
& = & \Big(\sup \{v \in \mathcal{SH}_m(\Omega) \, ; \, v \leq 0 \, \, \text{and} \, \, v \leq - 1\,  \text{in} \, \, K\}\Big)^*. 
\end{eqnarray}
The main properties satisfied by the extremal function are summarized in the following result. 
\begin{lemma} \label{lem:extremalproperties} (\cite{Lu15}) Let $\Omega \Subset \Omega$ be a bounded $m$-hyperconvex domain and $K \subset \Omega$ a  non $m$-polar compact subset. 

Then the following properties hold :

$(1)$  $h_K  \in \mathcal{SH}_m(\Omega)$ and $-1 \leq h_K < 0$ in $\Omega$;

$(2)$ $h_K = -1$ quasi-everywhere in $K$;

$(3)$ $ (dd^c h_K)^m \wedge \beta^{n-m} = 0$ on $\Omega \setminus K$;

$(4)$ for any $\zeta \in \partial \Omega$, $\lim_{z \to \zeta} h_K(z) = 0$;

$(5)$ the capacity of the condenser $(K,\Omega)$ is positive and given by the formula 
\begin{equation} \label{eq:Capacity}
 c_{m}(K) = \int_{\Omega} (dd^c h_K)^m \wedge \beta^{n-m} = E_m(h_K).
\end{equation}
\end{lemma}
The last equality in the formula \eqref{eq:Capacity} follows from the fact that the equilibrium measure $\nu_K:=(dd^c h_K)^m \wedge \beta^{n-m} $  is supported on $K$ and $h_K = - 1$ quasi-everywhere in $K$, hence $\nu_K$-a.e. in $K$.

 \begin{lemma}\label{lem:FirstVariation} (\cite{Lu15})
The following properties hold :
 
$(1)$ The functional $E_m$ is G\^ateaux differentiable and $-E_m$ is a primitive of the complex $m$-Hessian operator i.e  for any smooth path $t \longmapsto \phi_t $ in $ \mathcal{E}^1_m(\Omega)$ defined in some interval $I\subset \R,$ we have 
 \begin{equation}\label{eqlem1}
 \frac{d}{dt}E_m  (\phi_t)=\int_{\Omega}(-\dot \phi_t)(dd^c\phi_t)^m\wedge \beta^{n-m}, \, t \in I.
 \end{equation}
 In particular if $u, v \in \mathcal{E}_m^1 (\Omega)$ and $u \leq v$, then  $E_m (v) \leq E_m (u)$; 
		
$(2)$ we have 
\begin{eqnarray*}
 \frac{d^2}{d t^2} E_m  (\phi_t) & = & \int_\Omega (- \ddot{\phi}_t) (dd^c \phi_t)^m \wedge \beta^{n-m}   \\
 & + &  m  \int_\Omega d \, \dot{\phi}_t \wedge d^c \dot{\phi}_t \wedge (dd^c \phi_t)^{m - 1}\wedge \beta^{n-m}, \, t \in I
 \end{eqnarray*}
 In particular, for any  $u, v \in \mathcal{E}^1_m (\Omega)$,  the function $t \longmapsto E_m( (1-t) u + t v)$ is a convex function in $[0,1]$;
		
$(3)$  the functional $E_m : \mathcal{E}^1_m (\Omega) \longrightarrow \R^+$ is lower semi-continuous on $\mathcal{E}^1_m (\Omega)$ for the $L^1_{loc}(\Omega)$-topology;

$(4)$  for any $C > 0$ the set $\mathcal{E}_m^1(\Omega, C)$ is compact for the $L^1_{loc}(\Omega)$-topology.
 \end{lemma}

\subsection{Stability theorems}
To investigate the continuity properties of the functional $I_{m,\mu,\Omega}$ we will need to consider a more special class of Borel measures introduced in \cite{K98} and considered in \cite{CZ23}. 

The following definition was introduced in  \cite{CZ23} following a classical terminology in Potential Theory (see \cite{Po16}). 
\begin{definition} \label{def:Diffuse}
1) A positive Borel measure on $\Omega$ is said to be diffuse with respect to the capacity $c_m = c_{m,\Omega}$
 if there exists a continuous increasing function $\Gamma : \R^+ \to \R^+$ such that $\Gamma (0) = 0$ and for any compact     
 subset $K \subset \Omega$ we have
 \begin{equation} \label{eq:diffuse}
 \mu (K) \leq  \Gamma(c_m(K)). 
 \end{equation}
 In this case we will say that $\mu$ is $\Gamma$-diffuse with respect to $c_m$.
 
 \smallskip
 
 2) A positive Borel measure on $\Omega$ is said to be strongly $\Gamma$-diffuse with respect to the capacity $c_m $ if it's $\Gamma$-diffuse with  $ \Gamma(t) = t \gamma (t)$, where $\gamma$  is a non decreasing continuous function satisfying  the following strong Dini type condition
 \begin{equation}\label{eq:Dini}
\int_{0^+}^1 \gamma(t)^{\frac{1}{m} } \dfrac{d t}{t} =  \int_{0^+}^1  \dfrac{\Gamma(t)^\frac{1}{m}}{t^\frac{1}{m}}  \dfrac{ d t }{t}   < +\infty.
 \end{equation}
\end{definition}
 The notion of domination by capacity was first considered in \cite{K98} when $m=n$.
 
 We define $\mathcal {M}_m(\Omega,\Gamma)$ to be the set of all positive Borel measures $\mu$  on $\Omega$ with finite mass  which are strongly $\Gamma$-diffuse.

Let us give some examples (see \cite{CZ23}).

\begin{example} 1. Let $0 \leq g \in L^p(\Omega)$ with $p>n\slash m$ and  $\mu := g dV_{2 n}$. Then by a result of Dinew-Ko\l odziej \cite{DK14}, $\mu \in \mathcal M_m(\Omega,\Gamma)$ for some function $\Gamma$ defined by $\Gamma (t) = A t^{\tau}$ with $\tau > 1$, which obviously satisfies the strong Dini condition \eqref{eq:Dini} (see also \cite[Example 3.4]{CZ23}).

2. Let $\mu$ be a Borel measure such that there exists $\varphi \in \mathcal SH_m(\Omega) \cap C^\alpha(\bar \Omega)$, $0 < \alpha \leq1$ with $\varphi = 0$ in $\partial\Omega$ such that $\mu \leq (dd^c \varphi)^m \wedge \beta^{n-m}$ in the sense of measures on $\Omega$. Then $\mu \in \mathcal M_m (\Omega,\Gamma)$ for some function $\Gamma$ by $\Gamma (t) = A t^{\tau}$ with $\tau > 1$, which satisfies the Dini condition \eqref{eq:Dini} (see \cite[Corollary 4.1]{CZ23}).
\end{example}

\smallskip

We will need the following weak stability theorem (\cite[Theorem 3.11]{CZ23}).
\begin{theorem} \label{thm:Stability} Let $\mu \in \mathcal{M}_m(\Omega,\Gamma)$ with $\Gamma$ satisfying the strong Dini condition \eqref{eq:Dini}. 
Let   $u, v\in \mathcal{SH}_m (\Omega) \cap L^{\infty} (\Omega)$ be  such that $ \liminf_{z\in\partial\Omega}(u-v)(z)\geq 0 $ and
$$
 (dd^c u)^m \wedge \beta^{n-m} \leq \mu,
$$
 in the sense of currents on $\Omega$. 
 
 Then
 $$
 \sup_{\Omega} (v-u)_+ \leq B h  (\Vert v - u)_+\Vert_{m,\mu}^m)
 $$
 where $\Vert (v - u)_+\Vert_{m,\mu}^m := \int_\Omega  (v - u)_+^m d \mu$, $B = B(m,\gamma,\mu(\Omega))> 0$ is a uniform constant   and $h$ is a continuous function on $\R^+$ such that $h(0)=0$ depending only on $\mu$ and $\Gamma$.
\end{theorem}

We will also need the following result which is a consequence of a combination of Theorem 1 and Theorem 2.2 in  \cite{CZ23}.

\begin{theorem}\label{thm:Continuity}
	Let $\mu \in \mathcal{M}_m(\Omega,\Gamma)$ with $\Gamma$ satisfying the strong Dini condition \eqref{eq:Dini}. Then for any $A >0$, there exists $C_m= C(m, \Omega,A) >0$ such that for any $u, v \in \mathcal E_m^1( \Omega, A)$
	
	\begin{equation}\label{eq:CZ}
		\int_{\Omega}\vert u-v \vert^{m+1} d \mu \leq C_m \tilde h \Big( \int_{\Omega}\vert u-v \vert^{m+1} d V_{2n}\Big),
	\end{equation}
	where $\tilde h=\tilde h_\Gamma$ is continuous in $\R^+$ with $h(0)=0.$
\end{theorem}

     \section{Non linear Sobolev-Poincaré type inequalities}
There are several versions of this type of inequalities (see \cite{AC20} and \cite{WZ22}).
Here we will give different type of inequalities more suitable for our approach.

  \subsection{Integrability of finite energy potentials}
  To define the functional $I_{m,\mu,\Omega}$ we need to consider measures that staisfy the integrability condition \eqref{eq:Integrability}.
  We give sufficient conditions on a  Borel measure $\mu$ to insure  integrability properties for  functions in the class $\mathcal E^1_m(\Omega)$.

  We will need the following result.
 \begin{lemma}\label{lem:Sob}
	Let $\Omega \Subset \C^n$ be a bounded $m$-hyperconvex domain and $\mu$ be a positive Borel measure on $\Omega$. Assume that there exists a function $v \in \mathcal{SH}_m(\Omega) \cap L^\infty(\Omega)$ satisfying 
	$$
	\mu \leq (dd^c v)^m \wedge \beta ^{n-m} ~\text{on}~ \Omega ~\text{ and }~ v_{|\partial \Omega}=0.
	$$
	Then there exists a constant $A>0$ such that for any $\phi \in \mathcal{E}^1_m(\Omega)$, we have 
	\begin{equation}\label{eq:Sob}
		\int_{\Omega}(-\phi)^{m+1} d\mu \leq A E_m(\phi).
	\end{equation}
	In particular $\mathcal{E}^1_m(\Omega) \subset L^{m+1} (\Omega,\mu)$.
  \end{lemma}

  Observe that for $m=1$ , we have 
  $$E_1(\phi) = \frac{1}{2}\int_\Omega (-\phi) dd^c \phi \wedge \beta^{n-1} =  \frac{1}{2}\int_\Omega d \phi \wedge d^c \phi \wedge \beta^{n-1}.$$
   Hence the inequality \eqref{eq:Sob} in this case is the classical Poincaré inequality for functions in $\mathcal E^1 (\Omega) \subset W_0^{1,2} (\Omega)$.

\begin{proof} 
To prove the inequality \eqref{eq:Sob}, we will need the following estimate due to Z. B\l{}ocki for the complex Monge-Ampère operator.

For any $u, w \in \mathcal E^0_m (\Omega)$ we have
\begin{equation} \label{eq:Blocki}
\int_\Omega (-u)^{m+1} (dd^c w)^m\ \wedge \beta^{n-m} \leq  {(m+1) !}  \, \Vert w\Vert^m_{L^\infty (\Omega)} \int_\Omega (-u) (dd^c u)^m \wedge  \beta^{n-m}.
\end{equation} 
This can be proved using integration by parts $m$ times  (see \cite{Bl93}).

To prove the estimate \eqref{eq:Sob}, it is enough to assume that $\phi \in \mathcal E^0_m (\Omega)$.
Then  since we have
$$
\int_\Omega (-\phi)^{m+1}  d\mu  \leq \int_\Omega (-\phi)^{m+1}  (dd^c v)^m \wedge  \beta^{n-m},
$$
it follows from \eqref{eq:Blocki} that
$$
\int_\Omega (-\phi)^{m+1}  d\mu  \leq  {(m+1) !}  \,  \Vert v\Vert^m_{L^\infty (\Omega)} \int_\Omega (-\phi) (dd^c \phi)^m\wedge  \beta^{n-m},
$$
which proves the required estimate with $A :=  (m+1) {(m+1)!}  \, \Vert v\Vert^m_{L^\infty (\Omega)}$.	
\end{proof}

We extend the previous result to a special class of diffuse measures.  
 \begin{proposition} \label{prop:DCI} Let  $\mu$ be a finite $\Gamma$-diffuse Borel measure with respect to  $c_m = c_{m,\Omega}$  i.e. for any Borel set $K \subset \Omega$,
 \begin{equation} \label{eq:DC}
 \mu (K) \leq   \Gamma(\mathrm{c}_m(K)).
 \end{equation}
  Assume that  there exists $r > 0$ such that $\Gamma$ satisfies the following integrability condition
  \begin{equation} \label{eq:ellGamma}
  \ell_\Gamma  := \int_0^1 \frac{\Gamma (t)}{t^{1 + r\slash (m+1)}}  d t < +\infty. 
  \end{equation}
  Then  there exists a constant $C = C (m,r,\ell_\Gamma,\mu(\Omega)) > 0$ such that for any $\phi \in \mathcal E^1_m (\Omega) $,
 $$
  \int_\Omega (-\phi)^r d \mu  \leq C  E_m(\phi)^{r \slash (m+1)}.
 $$
 In particular $\mathcal E^1_m (\Omega) \subset L^r (\Omega,\mu)$. 
 \end{proposition}
 \begin{proof}
 Again the idea of the proof is classical (see \cite{BJZ05}). Let $\phi \in \mathcal E^1_m (\Omega)$.  Then by Cavalieri-Fubini principle we have
 $$
 \int_\Omega (-\phi)^r d \mu \leq  \mu (\Omega) + r \int_1^{+\infty} t^{r-1} \mu (\{\phi < - t\}) d t.
 $$
 From the domination property we get
 $$
 \int_\Omega (-\phi)^r d \mu \leq  \mu (\Omega) + r \int_1^{+ \infty} s^{r-1} \Gamma (\mathrm{c}_m(\{\phi < - s\})) d s.
 $$
 By \cite[Lemma 7.1]{Lu15}, it follows that  there exists a uniform constant $B_{m} >0$ such that for any $s >0$,
 \begin{equation} \label{eq:CapacityEnergy}
 \mathrm{c}_m(\{\phi < - s\})  \leq \frac{B_{m}}{s^{m +1}} E_m(\phi),
 \end{equation}
 Assume first that $E_m(\phi) \leq 1 \slash B_m $.  Then 
 
 $$
 \int_\Omega (-\phi)^r d \mu \leq  \mu (\Omega)  + r \, \int_1^{+ \infty} s^{r} \Gamma \left(s^{- m-1}\right) \frac{d s}{s}.
 $$
 Set $t =  s^{-m-1}$. Then we get
$$
\int_\Omega (-\phi)^r d \mu \leq  \mu (\Omega) + r(m+1)  \int_0^{ 1} \frac{\Gamma (t)}{t^{1 + r\slash (m+1)}} d t. 
 $$ 
  Using \eqref{eq:ellGamma}, it follows that
 \begin{eqnarray*}
 \int_\Omega (-\phi)^r d \mu  & \leq & A,
 \end{eqnarray*}
 where $A := \mu (\Omega) +   r(m+1) \, \ell_\Gamma$. 
 
 Now if $E_m(\phi) \geq  1 \slash B_m$, we consider the function defined by 
 $$
 \tilde \phi := B_m ^{-{1\slash (m+1)}} E_m(\phi)^{-{1\slash (m+1)}}\phi.
 $$ 
 Then  $\tilde \phi \in \mathcal E_m^1(\Omega)$  and $E_m(\tilde \phi) \leq 1 \slash B_m$. 
  Applying the previous inequality to $\tilde \phi$, we get by homogeneity $\int_\Omega (-\phi)^r d \mu  \leq C E_m (\phi)^{r\slash m}$, where $C := A B_m^{r\slash (m+1)}$ which is a uniform constant. This proves the required  inequality.
 \end{proof}
 As a consequence we get the following result.
 \begin{corollary} \label{prop:DC} Let $\mu$ be a finite Borel measure which satisfies the following condition :  there exist constants $A > 0$ and $\tau > 0$ such that for any Borel set $K \subset \Omega$,
 \begin{equation} \label{eq:DC}
 \mu (K) \leq A  \, \, c_m(K)^{\tau}.
 \end{equation}
 Then  for any $0 <  r < \tau(m + 1)$, there exists a constant $B = B(m,\tau, r, A,\Omega) > 0$ such that
 $$
  \int_\Omega (-\phi)^r d \mu  \leq B E_m(\phi)^{r \slash (m+1)}.
 $$
 In particular $\mathcal E^1_m (\Omega) \subset L^r (\Omega,\mu)$. 
 \end{corollary}

 \begin{remark} It is easy to see that the previous sufficient condition for integrability is almost optimal. 
 Indeed, assume that $\mathcal E^1_m (\Omega) \subset L^r (\Omega,\mu)$. Then arguing by contradiction as in \cite{GZ07}, it's easy to deduce 
 that there exists a constant $B > 0$ such that for any $\phi \in  \mathcal E^1_m (\Omega)$, we have
 $$
  \int_\Omega (-\phi)^r d \mu  \leq B E_m(\phi)^{r \slash (m+1)}.
 $$

Let $K \subset \Omega$ be a compact set. Applying this inequality to  the extremal function 
$h_K = h_{K,\Omega} = P_{m,\Omega} (- {\bf 1}_K)$ and taking into account the properties of $h_K$  and the formula \eqref{eq:Capacity} in Lemma \ref{lem:extremalproperties}, we deduce that
$$
\mu(K) \leq \int_\Omega (-h_K)^r d \mu  \leq B \, E_m(h_K)^{r \slash (m+1)} \leq B \, c_m(K)^{r \slash (m+1)}.
$$
 \end{remark}
 
 The following corollary was proved in \cite{AC20} by similar methods.
 \begin{corollary} \label{coro:Integrability} Let $\mu := f \beta^n$ where  $0 \leq f \in L^p (\Omega)$ with $p > n\slash m$.  Set $k (m,n,p) := n (p-1) \slash p (n-m)$. 
 Then for any $1 < r  <  (m+1)  k (m,n,p)$, there exists a constant $C = C(r, m,n, p,\Vert f\Vert_p) >0$ such that
 \begin{equation}\label{eq:poincare2}
 \int_\Omega (- \phi)^{r} d \mu \leq C \,  E_m(\phi)^{r\slash (m+1)}.
 \end{equation}
 In particular   $ \mathcal E^1_m (\Omega) \subset L^{r} (\Omega,\mu)$.
 \end{corollary} 
 \begin{proof} This follows from the comparison of volume and $m$-capacity due to Dinew-Ko\l odziej (\cite{DK14}). Namely for any $\theta < n\slash (n-m)$, there exists a constant $M = M(\theta,\Omega) > 0$ such that
 \begin{equation}
 \int_K \beta^n \leq M c_m(K)^\theta.
 \end{equation}
 By H\"older inequality, it follows that
 $$
 \int_K f \beta^n \leq M^{1 \slash q} \Vert f \Vert_p (c_m(K))^{\theta \slash q},
 $$
 where $q := p\slash (p-1)$. 
 This means that the domination condition \eqref{eq:DC} is satisfied with any $\tau < k (m,n,p) := \frac{n (p-1)}{ p (n-m)}$.
 The conclusion follows then from the  previous Corollary.
 \end{proof}

If $\mu$ satisfies the assumptions of Lemma \ref{lem:Sob} or Proposition \ref{prop:DCI} with $r \geq m +1$, then we have 
$$
\mathcal{E}_m^1(\Omega) \subset L^{m+1}(\Omega,\mu).
$$ 
In this case the functional 
$$
I_{m,\mu, \Omega}(\phi):=\frac{1}{m+1}\int_{\Omega}(-\phi)^{m+1} d\mu
$$
is well defined on $\mathcal{E}_m^1(\Omega)$. \\

Now we investigate the continuity properties of $I_{m,\mu,\Omega}$. 

\begin{theorem}\label{thm:continuity}
Let $\Omega \Subset \C^n$ be a bounded  $m$-hyperconvex domain and $\mu \in \mathcal{M}_m(\Omega,\Gamma)$ with $\Gamma$ satsifying the strong Dini type condition \eqref{eq:Dini}. 

Then we have the following properties :

 1) $\mathcal E^1_m (\Omega) \subset L^{m+1} (\Omega,\mu)$;
	
 2) for each $C > 0$ and  any  sequence $(u_j)$ in $\mathcal{E}^1_m (\Omega,C)$ converging to $u \in \mathcal{E}^1_m (\Omega,C)$ in $L^1_{loc}(\Omega)$,  the sequence  $(u_j)$ converges to $u$ in $L^{m+1} (\Omega,\mu)$.
	
 In particular  the functional $I_{m,\mu,\Omega}$ is continuous on $\mathcal E^1_m(\Omega,C)$ for the $L^1_{loc}(\Omega)$-topology and 
    
    \begin{equation}  \label{eq:Formula1}
    \lim_{j \to + \infty} \int_\Omega (-u_j)^{m+1}  d \mu = \int_\Omega (-u)^{m+1}  d \mu.
    \end{equation}
    \end{theorem}

    \begin{proof} 1) We can apply   Proposition \ref{prop:DCI}  with $r = m +1$. Indeed if $\Gamma (t) = t \gamma (t)$,  the condition \eqref{eq:ellGamma} means that $\int_0^1 t^{-1} \gamma (t) dt < + \infty$,   which is verified if $\gamma$ satisfies the Dini type condition \eqref{eq:Dini}. This shows that $\mathcal E^1_m (\Omega) \subset L^{m+1} (\Omega,\mu)$.
    
  2)  Let $(u_j)$ be a sequence in $\mathcal{E}^1_m (\Omega,C)$ converging to $u \in \mathcal{E}^1_m (\Omega,C)$ in the $L^1_{loc}(\Omega)$-topology. 
  
  Assume first that $(u_j)$ is uniformly bounded in $\Omega$. Then since the sequence $(u_j)_{j \in \N}$ converges to $u$ in $L^1_{loc}(\Omega; dV_{2n})$, taking a subsequence if necessary,  we can assume that $u_j \to u$ a.e. in $\Omega$ with respect to the Lebesgue measure in $\Omega$. It follows  from the Lebesgue convergence theorem  that $ u_j \to u $ in $L^{p}(\Omega, d V_{2n})$ for any $p>1.$ Hence $ u_j \to u $ in $L^{m+1}(\Omega, d V_{2n})$ and by Theorem \ref{thm:Continuity}, it follows that $u_j   \to   u $ in $L^{m+1} (\Omega, \mu)$. This implies the formula \eqref{eq:Formula1}.    
	
  We now consider the general case.  For  fixed $k, j  \in \N$,  we define $u^{(k)} := \sup \{u,-k\}$ and 
  $ u_j^{(k)} :=  \sup \{u_j ,  -  k\}$.  We also define for $j, k \in \N$, $h_j := (-u_j)^{m+1}$,  $h_j^{(k)} = (-u_j^{(k)})^{m+1}$ and  
  $h := (-u)^{m+1}$ and $h^{(k)} :=(- u^{(k)})^{m+1} $. These are Borel  functions in  $L^1(\Omega, \mu)$ and  we have the 
  following obvious inequalities :
	\begin{eqnarray} \label{eq:FundEq3}
	\Big\vert\int_\Omega (h_j - h)  d \mu \Big\vert &\leq  & \int_\Omega (h_j^{(k)} - h_j) d \mu + \int_\Omega \vert h_j^{(k)} - h^{(k)}\vert 
	d \mu\\
	& + &    \int_\Omega (h^k -h)  d \mu. \nonumber
	\end{eqnarray}
	
	For fixed $k$, the sequence  $(u_j^{(k)})_{j \in \N}$ is a  uniformly bounded sequence of $m$-subharmonic functions in $\Omega$.  Then applying the first step, we see that for each $k \in \N$, the second term in (\ref{eq:FundEq3}) converges to $0$ as $j \to + \infty$, while the third term converges to $0$ by the monotone convergence theorem when $k \to + \infty$.  It remains to show that the first term converges to $0$ as $k \to + \infty$, uniformly in $j$.
	Indeed,  for $j, k \in \N^*$ we have the following obvious estimates 
	\begin{equation} \label{eq:1}
	\int_\Omega\vert h_j - h_j^{(k)} \vert d \mu \leq 2 \int_{\{h_j \geq k^{m+1}\}} h_j d \mu. 
	\end{equation}
	We claim that  the sequence $k \longmapsto \int_{\{h_j \geq k^{m+1}\}} h_j d \mu$ converges to $0$ uniformly in $j$ as $k \to + \infty$.
	Indeed  for fixed $j, k$, we have
	\begin{equation} \label{eq:2}
	\int_{\{h_j \geq k^{m+1}\}} h_j d \mu =  \int_{\{u_j \leq - k\}} (-u_j)^{m+1} d \mu.
	\end{equation}
	
	On the other hand, fix a Borel subset  $B \subset \Omega$. Since $\mu \in \mathcal{M}_m(\Omega, \Gamma)$, with $\Gamma$ satisfying the strong Dini condition \eqref{eq:Dini},
 it follows from \cite[Theorem 1]{CZ23} that there exists a function $\phi_B \in \mathcal{SH}_m (\Omega) \cap C^0(\bar \Omega)$ such that $\phi_B =0$ in $\partial \Omega$ and $(dd^c \phi_B)^m \wedge \beta^{n-m} = {\bf 1}_B \, \mu$ in the sense of currents on $\Omega$.
	
 Therefore as before,  Blocki's inequality \eqref{eq:Blocki} yields for any $j\in \N$,
 \begin{eqnarray*}
 \int_B (- u_j)^{m+1}  \mu & = & \int_\Omega (- u_j)^{m+1} (dd^c \phi_B)^m \wedge \beta^{n-m} \\ 
 & \leq& (m +1)! \, \Vert \phi_B \Vert^m_{L^{\infty} (\Omega)}  \int_\Omega (-u_j) (dd^c u_j)^m \wedge \beta ^{n-m}   \\
 & \leq & (m+1)! \, C_0 \Vert \phi_B \Vert_{L^{\infty} (\Omega)}^m,
 \end{eqnarray*}
where $C_0>0$ is a uniform constant.

By Lemma \ref{lem:uniforcv} below, we have that $\Vert \phi_B \Vert_{L^{\infty} (\Omega)} \to 0$ as $\mu (B) \to 0.$

This implies that $\sup_{j \in \N} \int_B (-u_j)^{m+1} d\mu \to 0$ as $\mu (B) \to 0$.
We want to apply this result to the Borel sets $ B_{j,k}:=  \{u_j \leq - k\} $. To estimate the mass of $\mu$ on the sets $B_{j,k}$, we first observe by using the inequalities \eqref{eq:diffuse} that for any $j, k \in \N$, we have

$$
\mu(\{u_j \leq - k\}) \leq \Gamma (c_m (\{u_j \leq - k\})).
$$

By \eqref{eq:CapacityEnergy}, we have for any $j, k \in \N^*$,

$$
c_m ( \{u_j \leq - k\} ) \leq \frac{D_0}{k^{m+1}} E_m(u_j) \leq \frac{D_0 M}{k^{m+1}}=: \varepsilon_k, 
$$
and $\varepsilon_k \to 0$ as $k \to \infty$. Therefore 

$$
\sup_{j \in \N}  \mu ( \{u_j \leq - k\}) \leq \Gamma (\varepsilon_k ) \to 0, \, \text{as} \, k \to + \infty.
$$
This proves the claim and completes the proof of  the Theorem.
\end{proof}

Now we prove the lemma used in the last proof. 

\begin{lemma}\label{lem:uniforcv}
	Let $\Omega \Subset \C^n$ be a bounded strictly $m$-pseudoconvex domain and $\mu \in \mathcal{M}_m(\Omega,\Gamma)$ with $\Gamma$  satisfying the strong Dini  condition \eqref{eq:Dini}. 
	
	Then for any Borel subset $B \subset \Omega$, there exists a unique function $\phi_B \in \mathcal{SH}_m(\Omega) \cap C^0( \bar \Omega)$ such that $\phi_B =0$ in $\partial \Omega$ and $(dd^c \phi_B)^m \wedge \beta^{n-m} = {\bf 1}_B \, \mu$ in the sense of currents on $\Omega$.
	
	Moreover $\Vert \phi_B \Vert_{C^{0} (\bar \Omega)} \to 0$ as $\mu (B) \to 0.$
\end{lemma}

\begin{proof} The existence of $\phi_B$ follows from \cite[Theorem 1]{CZ23}. It remains to prove the second part of the lemma.
	
 We argue by contradiction. Assume that there exists a sequence of Borel sets $(B_j)_{j\in \N ^*}$ in $\Omega$ such that $\mu (B_j) \to 0$ as $j \to + \infty$ and $ \Vert \phi_{B_j} \Vert_{L^{\infty} (\Omega)}\geq \delta_0,$ where $\delta_0 >0$ is a uniform constant and $\phi_{B_j} \in \mathcal{SH}_m(\Omega) \cap C^0( \bar \Omega)$ satisfies $\phi_{B_j}=0$ in $\partial \Omega$ and $(dd^c \phi_{B_j})^m \wedge \beta^{n-m} ={\bf 1}_{ B_j} \, \mu$ on $\Omega.$ Without lost of generality we can assume that $\mu (B_j) \leq 2^{-j}$ for any $j\in \N$.
	
 Assume first that $(B_j)_j $ is a decreasing  sequence and let $B := \cap_{j} B_j$. Then $\mu(B)=0$ and for any $j \in \N$, we have ${\bf 1}_{B_{j+1} } \mu \leq {\bf 1}_{B_j} \mu$ on $\Omega.$ This implies that
	
 $$
 (dd^c \phi_{B_{j+1}})^m \wedge \beta^{n-m} \leq (dd^c \phi_{B_j})^m \wedge \beta^{n-m} \, \, \text{ on} \, \, \Omega,
 $$
with $\phi_{B_{j+1}}=\phi_{B_j}$ in $\partial \Omega.$ By the comparison principle, $(\phi_{B_j})_{j\in \N^*}$ is a increasing sequence in $\Omega$ which converges a.e to $\phi \in \mathcal{ SH}_m( \Omega) \cap L^{\infty}(\Omega)$ with $\phi_{|\partial \Omega}=0.$

Since the Hessian operator is continuous w.r.t the monotone convergence and $\mu(B) =0$, it follows that 

$$
(dd^c \phi_{B_j})^m \wedge \beta^{n-m} \to (dd^c \phi)^m \wedge \beta^{n-m}= {\bf 1}_B \mu=0.
$$
Since $\phi_{|\partial \Omega}=0$, it follows from the comparison principle that $\phi=0$ in $\Omega.$

Now we prove that $\Vert \phi_{B_j} \Vert_{L^{\infty} (\Omega)} \to 0$  as $j \to + \infty.$ 

Since $\mu \in \mathcal{M}_m(\Omega,\Gamma)$, by the weak uniform stability Theorem \ref{thm:Stability}, there exists a uniform contant $C_0 > 0$ such that 
 \begin{equation} \label{eq:StabIneq}
 \Vert \phi_{B_j} \Vert_{L^{\infty} (\Omega)} \leq C_0 h \Big(\Vert  \phi_{B_j} \Vert^m_{L^{m} (\Omega, \mu)}\Big),
\end{equation} 
 where $h(t) \to 0$ as $t \to 0.$ 
	
 Since $(\phi_{B_j})_j \nearrow \phi =0$ a.e in $\Omega$, we have  $(\sup \{\phi_{B_j} ; j \in \N\})^* = \phi=0$ in $\Omega$. Since we know that $c_m(\{\sup \{\phi_{B_j} ; j \in \N\} <  0\})= 0$ (see \cite{Lu15}) and $\mu$ is diffuse w.r.t the $c_m$-capacity, it follows that  $\mu(\{\sup \{\phi_{B_j} ; j \in \N\} < 0\}) = 0$. This  implies that $(\phi_{B_j})_j \nearrow 0$ $\mu$-a.e in $\Omega$ as $j \to + \infty$. Hence by the monotone convergence theorem, the right hand side in \eqref{eq:StabIneq} converges to $0$ as $j \to + \infty$ and then  $\Vert \phi_{B_j} \Vert_{L^{\infty} (\Omega)} \to 0$ as $j \to + \infty$. This contradict the fact that
    $ \Vert \phi_{B_j} \Vert_{L^{\infty} (\Omega)} \geq \delta_0.$

The general case is easily deduced from the first case by setting for any 
 $j \in \N$, $A_j := \bigcup_{k \geq j} B_k$. Then $(A_j)_j$ is a decreasing sequence of Borel sets which decreases to a Borel set $A $. By sub-additivity  we have $ \mu(A_j) \leq 2^{-j+1}$. 
Moreover  since $B_j \subset A_j$, by the comparison principle we have $\phi_{A_j} \leq \phi_{B_j} \leq 0$ in $\Omega$, hence $\Vert \phi_{A_j} \Vert_{L^{\infty}(\Omega)} \geq \Vert \phi_{B_j} \Vert_{L^{\infty}(\Omega)} \geq \delta_0$ for any $j \in \N$. Applying the previous reasoning with $A_j$ instead of $B_j$ we obtain a contradiction.
\end{proof}
 \section{The eigenvalue problem} 

\subsection{ The variational approach : Proof of Theorem 1.1}
  In this section we use a variational method to solve the    eigenvalue  problem stated in the introduction for a twisted complex Hessian operator.
  \begin{proof}  We first observe that since $\Omega$ is $m$-hyperconvex, it admits a bounded negative $m$-subharmonic exhaustion $\rho$ such that $\int_\Omega (dd^c \rho)^m \wedge \beta^{n-m} < \infty$.
  Indeed fix a subdomain $\Omega_0 \Subset \Omega$ and set $K := \bar{\Omega}_0$.  Hence $0 < c_m(K) < + \infty$ and by Lemma \ref{lem:extremalproperties}, the function $\rho := h_K$ defined by the formula \eqref{eq:extremalfonction} is such an exhaustion of $\Omega$.
   Then  $\rho \in \mathcal E^1_m(\Omega)$ and any  $w \in \mathcal E^1_m(\Omega)$ such that $w <  0$  satisfies  $\int_\Omega (-w)^{m +1} d \mu > 0$. Indeed, since $\mu(\Omega) > 0$ there exists a compact set $K \Subset \Omega$ such that $ \mu (K) := \int_K d \mu >0$. Therefore   $\int_\Omega (-w)^{m +1} d \mu  \geq  (- \max_K w)^{m +1} \mu (K) > 0$. Hence  $\lambda_1$ is a  well defined non negative real number and by homogeneity, we have
	\begin{equation} \label{eq:VariatFormula2}
	\lambda_1^{m} =   \inf \{ E_m (w) \,  ; \,  w \in \mathcal E^1_m (\Omega), \,   I_m  (w) = 1\},
	\end{equation} 
 where we set $I_m (\phi) = (m+1)^{-1} \int_\Omega (-\phi)^{m+1} d \mu$.
	
Moreover, by  Lemma \ref{lem:Sob}, there exists a constant $A > 0$ such that for any $w \in \mathcal E^1_m (\Omega)$
$$
I_m (w) \leq A \, E_m(w).
$$

In particular  we  conclude that $\lambda_1^m \geq A^{-1} >0$.

On the other hand, by definition there exists a minimizing sequence $(w_j)_{j \in \N}$ in $\mathcal E^1_m (\Omega)$ such that $  I_{m}   (w_j) = 1$ for any $j \in \N$ and
$$
\lim_{j \to + \infty} E_m (w_j) = \lambda_1^m.
$$

By construction, the sequence $(E_m(w_j))_{j \in \N}$ is bounded. It follows  from Lemma  \ref{lem:Sob} applied to the Lebesgue measure $\lambda_{2n}$, that the sequence $(w_j)_{j \in \N}$ is bounded in $L^{m+1} (\Omega,\lambda_{2n})$. Extracting a subsequence if necessary we can assume that $(w_j)$ converges weakly to $w \in \mathcal{ SH}_m (\Omega)$ and a.e. in $\Omega$ (wrt the Lebesgue measure). Hence $(w_j)$ converges to $w$  in $L_{loc}^{1} (\Omega)$.  By lower semi-continuity of the energy functional, it follows that $w \in \mathcal E^1_m (\Omega)$ and 

$$
\displaystyle{ E_m (w) \leq \liminf_{j \to + \infty} E_m (w_j) = \lambda_1^m}.
$$

Since $\sup_j E_m (w_j) =: C < + \infty$, it follows from  Lemma \ref{lem:Sob} that  
\begin{equation} \label{eq:FundEq1}
\lim_{j \to + \infty} \int_\Omega(-w_j)^{m+1} d \mu =  \int_\Omega(-w)^{m+1} d \mu.
\end{equation}
Hence  $I_{m} (w) = 1$ and 
$w\in \mathcal E^1_m (\Omega)$ is an "extremal" function  for the eigenvalue problem i.e.
$$
\lambda_1^m = \frac{E_m (w)}{I_{m}  (w)}\cdot 
$$

 Since $I_{m}  (w) = 1$ it follows that $u_1 := w \not \equiv 0$ in $\Omega$. To prove that $(\lambda_1,u_1)$ is a solution to the eigenvalue problem, consider the following functional defined for $\phi \in \mathcal E^1_m (\Omega)$, by the formula 
$$
\Phi_{m} (\phi) := E_m (\phi) - \lambda_1^m I_{m} (\phi), 
$$
and observe that when $\phi$ is smooth then 
$$
\Phi_{m}' (\phi) = - (dd^c \phi)^m \wedge \beta^{n-m}  +  \lambda_1^m (-\phi)^m  \mu,
$$
This means that the eigenvalue equation is the Euler-Lagrange equation of the functional $\Phi_m $ on $\mathcal E^1_m(\Omega)$.
Therefore it's enough to minimize the functional $\Phi_m$ on  $\mathcal E_m^1(\Omega)$.

As observed before, for any $\phi \in \mathcal E_m^1 (\Omega)$ with $\phi \not \equiv 0$, we have $ I_{m} (\phi) >  0$ and then
$$
\Phi_m (\phi) :=  I_{m}  (\phi) \left(\frac{E_m (\phi)}{I_{m}  (\phi)}  - \lambda_1^m\right) \geq 0, 
$$
by definition of   $\lambda_1. $ Since  $\Phi_m (u_1) = 0$,
this means that the functional $\Phi_m$ achieves its minimum on  $\mathcal E_m^1 (\Omega)$ at $u_1$.
Therefore $u_1$ is a kind of "critical point" of the functional $ \Phi_m$. To prove this claim, we will use a tricky argument which goes back to \cite{BBGZ13}.

Fix a "test function" $\psi \in \mathcal{E}_m^0(\Omega)$ and consider  the path 
$\phi_t = u_1+ t \psi$ which belongs to $\mathcal E_m^1 (\Omega)$ when $0 \leq t \leq 1$ by convexity. However when $t < 0$,  this is no longer the case, and so we consider its $m$-subharmonic envelope $\tilde \phi_t := P_m (\phi_t)$ in $\Omega$ defined by the formula 
$$
P_m (\phi_t):=\sup \{ v \in \mathcal{E}^1_m(\Omega) \, | \, v \leq \phi_t  \}\cdot
$$
Since $ u_1 \leq \phi_t $ when $t < 0$, it follows that $u \leq P_m(\phi_t)$ when $t < 0$, hence     $\tilde \phi_t \in \mathcal{E}^1_m (\Omega)$ for any $t \in [-1,+1]$ (see \cite{Ceg98, Lu15}).

Now consider the one variable function defined for $t \in [-1,+1]$ by 
$$
h(t) :=  E_m \circ P_m (\phi_t ) - \lambda_1^m I_{m} (\phi_t).
$$
We claim that the function $h$ is differentiable in $ [-1,+1]$, non negative and attains its minimum at $t = 0$. Indeed observe first that $h(0) =0$.
Moreover since for any $t \in [-1,1]$,  $ \tilde \phi_t  \leq \phi_t < 0$ in $\Omega$, it follows that $  I_{m} (\phi_t) \leq  I_{m} (\tilde \phi_t)$ and then
$$
h (t) \geq E_m ( \tilde \phi_t)  - \lambda_1^m I_{m} (\tilde \phi_t) = \Phi (\tilde \phi_t) \geq 0,
$$
for any $t \in [-1,1]$, which proves our claim.
By Lemma \ref{lem:FirstVariation}, we have 
$$
\frac{ d}{d t}  (E_m \circ P_m) (\phi_t) = \int_{\Omega}(- \dot  \phi_t) (dd^c P_m( \phi_t))^m \wedge \beta^{n-m}.
$$

Therefore $h$ is differentiable in $[-1,+1]$ and  by Lemma \ref{lem:FirstVariation}, we have for any $t \in [-1,1], $
\begin{eqnarray*}
	h'(t) & = & \frac{d}{d t} E_m (\tilde \phi_t) - \lambda_1^m  \frac{d}{d t}  I_{m}  ( \phi_t) \\
	& = & \int_\Omega (- \dot  \phi_t) (dd^c \tilde \phi_t)^m \wedge \beta^{n-m}   +  \lambda_1^m  \int_\Omega  \dot \phi_t (-\phi_t)^m d \mu. 
\end{eqnarray*}

Since $h$ achieves its minimum at $0$, it follows that $h'(0) = 0$, which implies the following identity: 
$$
\int_\Omega \psi \, (dd^c u_1)^m \wedge \beta^{n-m}  =  \lambda_1^m  \int_\Omega  \psi \, (-u_1)^m d \mu, 
$$
for any  $\psi \in \mathcal E^0_m(\Omega) $.  Since any smooth test function $\chi$ in $\Omega$ can be written as $\chi = \psi_1 - \psi_2$, where $\psi_1, \psi_2 \in \mathcal E^0_m (\Omega) $ (see \cite[Lemma 3.10]{Lu15}), it follows that 
$$
(dd^c u_1)^m \wedge \beta^{n-m}  = \lambda_1^m  (-u_1)^m  \mu. 
$$
in the  sense of currents on $\Omega$.

Now assume that $\Omega$ is strictly $m$-pseudoconvex and $(m,p)$ satisfies the conditions \eqref{eq:Holder}. To prove H\"older continuity, observe that  $(dd^c u_1)^m \wedge \beta^{n-m} = (-\lambda_1 u_1)^m g \beta^n$
has a density $g_1 :=  (-\lambda_1 u_1)^m g $ and by Lemma \ref{lem:Integ} below we have $g_1 \in L^r (\Omega)$ for some $r > n\slash m$. 

By \cite{Ch16} there exists  $v \in \mathcal{SH}_m(\Omega) \cap C^{\alpha} (\bar \Omega)$ for some $\alpha \in ]0,1[$  such that $(dd^c v)^m \wedge \beta^{n-m} = g_1 \beta^n$ on $\Omega$ and $v\equiv 0$ in $\partial \Omega$. 
Now we have two solutions $v, u_1 \in \mathcal E_m^1(\Omega)$ of the complex Hessian  equation $(dd^c \phi)^m \wedge \beta^{n-m} = g_1 \beta^n$. By uniqueness in $\mathcal E_m^1(\Omega)$ it follows that $v= u_1$, hence $u_1 \in C^{\alpha} (\bar \Omega)$ and $u_1 \equiv 0$ in $\partial \Omega$ (see \cite[Theorem 1.1]{Lu15}).
\end{proof}
Let us prove the lemma used in the previous proof.
\begin{lemma} \label{lem:Integ}  Let  $\Omega \Subset \C^n$ be a $m$-hyperconvex domain, $0 \leq g \in L^p(\Omega)$ with  $(m,p)$ satisfying the following conditions : 
\begin{equation} \label{eq:HolderHessian}
 (n-1) \slash 2 < m \leq n  \, \, \, \text{ and} \, \, \, \, p > p^*(m,n),
\end{equation}
where $p^*(m,n) \geq n\slash m$ is given by the formula \eqref{eq:Conditionp*} below.

 Then  there exists an exponent $r $  depending only on $p, m$ and $ n$,  such that   $n\slash m  < r < p$ and for any 
$\phi \in \mathcal E_m^1(\Omega)$, $g_\phi := (-\phi)^m g \in L^{r}(\Omega)$.
\end{lemma}
\begin{proof}
Fix $\phi \in \mathcal E_m^1(\Omega)$ and $r >1$ such that $1 < r <p$ and set $s := p\slash r >1$. 
 By H\"older inequality, we have
\begin{eqnarray*}
\int_\Omega g_\phi^r \beta^n &=& \int_\Omega  (-\phi)^{m r} g^{r}  \beta^n\\
& \leq & \left(\int_\Omega g^{rs}  \beta^n\right)^{1 \slash s}   \left(\int_\Omega  (-\phi)^{m rs'} \beta^n\right)^{1 \slash s'} \\
& \leq & \Vert g\Vert_{L^p(\Omega)}^{p\slash s} \left(\int_\Omega  (-\phi)^{\tau} \beta^n\right)^{1 \slash s'} 
\end{eqnarray*}
where $s' := s \slash (s-1)$ and  $\tau = \tau (r,p,m) := m rs'  = mrs \slash (s-1)$. 
Since $s = p\slash r$, we have   $ \tau (r,p,m) = m r p \slash(p-r)$.

From this inequality we see that $g_\phi \in L^r(\Omega)$ if $\phi \in L^\tau(\Omega)$.
By Corollary \ref{coro:Integrability} applied to the Lebesgue measure,  $\phi \in L^{\tau} (\Omega)$  if the following condition holds 
\begin{equation} \label{eq:Conditionr}
\tau = \tau (r,p,m) = \frac{m r p}{p-r}  < \frac{ n (m+1) (p-1)}{ p(n-m)}.
\end{equation}
Observe that if $m = n$ this condition is always satisfied for any exponent $r \in ]n\slash m,p[$.

Assume now that $m < n$. We want to find a condition on $(m,n,p)$ such that  the condition \eqref{eq:Conditionr} is satisfied with an exponent $r \in ]n\slash m ,  p[$.

Observe that the function $\tau$ is decreasing in $p$ and the right hand side is less than 
$\frac{ n (m+1)}{n-m}$. Therefore  a necessary condition for \eqref{eq:Conditionr} to hold with $r > n\slash m$ is that $m r < (m+1)n  \slash (n-m)$ and the real number 
\begin{equation} \label{eq:Conditionk}
 \ell = \ell (m,n) := (m+1)  \slash (n-m) > 1,
 \end{equation}
 which means that $m > (n-1)\slash 2$.

Now assume that this condition holds so that $\ell > 1$ and observe that the function $\tau$ is increasing in $r$.
Therefore by continuity there exists an $r \in ]n\slash m,p[$ so that \eqref{eq:Conditionr} holds if and only if it holds with $r = n\slash m$ i.e. 
\begin{equation} \label{eq:ConditionTau}
\tau (n\slash m,p,m) <  \frac{ \ell n (p-1)}{ p}.
\end{equation}  
The inequality \eqref{eq:ConditionTau} is equivalent to the following one
$$
m (\ell-1) p^2 - \ell(n + m) p + \ell n > 0.
$$
This is a quadratic polynomial in $p$ whose discriminant is positive.
Hence it has two zeros $p_* < p^*$ so that it is positive when $p > p^*$, where
\begin{equation} \label{eq:Conditionp*}
p^*= p^*(m,n) := \frac{ \ell(n+m) + \sqrt{\Delta}}{2m(\ell-1)}, \, \, \text{and} \, \, \, \Delta := \ell^2 (n-m)^2 + 4 \ell m n.
\end{equation}
 It's easy to check that $p^* > n\slash m$ which proves the Lemma. 
 
 Observe that if $m \to n$ then $\ell \to \infty$ and $p^*(m,n)\to 1$.
\end{proof}

\subsection{The monotonicity property of $\lambda_1 (\Omega)$}
Let $\Omega' \Subset \Omega \Subset \C^n$ be two bounded  $m$-hyperconvex domains and $\mu \in \mathcal M (\Omega,\gamma)$. Then we have the following comparison theorem
\begin{theorem} Let $\mu' := {\bf 1}_{\Omega'} \mu$ be the restriction of $\mu$ to $\Omega'$. Then we have 
$$
\lambda_1 (\Omega,\mu) \leq \lambda_1 (\Omega',\mu').
$$
\end{theorem}
This  result was proved in \cite{BaZe23} for the eigenvalues of the  complex Monge-Ampère operator when $\mu := g dV$ with $g \in L^p(\Omega)$ ($p >1$).
\begin{proof}
We denote by $E_{\Omega} (u) := \frac{1}{m+1} \int_\Omega (-u) (dd^c u)^m \wedge \beta^{n-m}$ for $u \in \mathcal E_m^1(\Omega)$ and 
$I_{\Omega,\mu} (u) :=  \frac{1}{m+1} \int_\Omega (- u)^{m +1} d \mu$.

By Theorem \ref{thm:Variational1}  for $(\Omega',\mu')$ there exists $w' \in \mathcal E_m^1(\Omega')$ such that
$$
\lambda_1 (\Omega',\mu') = \frac{E_{\Omega'}(w')}{I_{\Omega',\mu'} (w')}.
$$
By the subextension Theorem (see \cite{CKZ11} for $m=n$), it follows that there exists $w \in \mathcal E_m^1(\Omega)$ such that $w \leq w'$ in $\Omega'$ and $ E_{\Omega}(w) \leq E_{\Omega'}(w')$. Since $w \leq w'$ in $\Omega'$, it follows that $I_{\Omega',\mu'} (w') \leq  I_{\Omega,\mu} (w)$, hence
$$
 \frac{E_{\Omega'}(w')}{I_{\Omega',\mu'} (w')} \geq  \frac{E_{\Omega}(w)}{I_{\Omega,\mu} (w)}.
 $$
 By Theorem \ref{thm:Variational1} for $(\Omega,\mu)$ we conclude that $ \lambda_1 (\Omega',\mu') \geq \lambda_1 (\Omega,\mu).$
\end{proof}

\smallskip


\smallskip

\section{A variational approach to more general equations}

\subsection{A general Dirichlet theorem} We consider the following more general Dirichlet problem
\begin{equation}\label{eq:DirPro}
\left\{\begin{array}{lcl} 
 (dd^c u)^m \wedge \beta^{n - m} = G (\cdot,u) \mu  &\hbox{on}\  \Omega, \\
  u = 0 & \,  \hbox{in}\  \partial \Omega, \\
 u <  0\, \, \, \, & \hbox{in}\  \Omega,
\end{array}\right.
\end{equation}
where $\Omega \Subset \C^n$ is bounded domain, $ G : \Omega \times ]- \infty, 0]  \longrightarrow [0,+ \infty[$ is a given Borel function, $\mu$ a positive Borel measure on $\Omega$ satisfying some conditions and $u \in \mathcal {SH}_m(\Omega) \cap L^{\infty}(\Omega)$.  

 We define  $H : \Omega \times \R^- \longrightarrow \R^+$ as follows
$$
H (z,t) := \int_t^0 G(z,s) d s,
$$ 
for $(z,t) \in \Omega \times ]-\infty,0]$.

We consider the following hypotheses : 
\begin{itemize}
\item $(H0)$ $\Omega$ is a bounded $m$-hyperconvex domain and $\mu$ is a positive Borel measure on $\Omega$ such that  $(\Omega,\mu)$ is  strongly $\Gamma$-diffuse in the sense of Definition \ref{def:Diffuse}.

\item $(H1)$  for $\mu$-a.e. $z \in \Omega$, the function $t \longmapsto G(z,t)$ is  continuous  on $]-\infty,0]$;
\item $(H2)$ there exists $\theta_1< \lambda_1^m \slash (m +1)$ such that
$$
\limsup_{t \to - \infty} \frac{\left(\sup_{z \in \Omega} H(z,t)\right)}{\vert t\vert^{m + 1}} < \theta_1 < \lambda_1^m \slash (m +1),
$$

\item $(H3)$ If $G (z,0)\equiv 0$ in $\Omega$,  there exists $\theta_2 > \lambda_1^m \slash (m +1)$ such that  
$$
\liminf_{t \to 0^-} \frac{ H (z,t)}{\vert t\vert^{m +1}} > \theta_2 > \lambda_1^m \slash (m +1),
$$ 
for  $\mu$-a.e. $z \in \Omega.$
\end{itemize}
Here $\lambda_1 := \lambda_1(\Omega,\mu)$ is the first eigenvalue of the twisted complex Hessian operator defined by the formula \eqref{eq:VP-Hessian}.

\smallskip
We define the corresponding functional on $\mathcal{E}^1_m (\Omega)$ by the formula
$$
 \Phi_{G,\mu} (\phi) := E_m (\phi)  - \int_\Omega H (z,\phi(z)) d \mu (z), \, \, \, \phi \in \mathcal{E}^1_m (\Omega)
$$

Formally the  Euler-Lagrange equation of the functional  $\Phi_{G,\mu} $ is precisely the Hessian equation (\ref{eq:DirPro}) as we will see.

We will prove the following result using the variational approach.

\begin{theorem}  \label{thm:Variational2} Assume that $(\Omega,\mu,G)$ satisfies the assumptions $(H_0)$, $(H1)$ and $(H2)$. Then the functional $\Phi_{G,\mu} $ has the following properties :

1) the functional  $\Phi_{G,\mu} $ is well defined and coercive on  $ \mathcal{E}^1_m (\Omega)$ and achieves its minimum   at a function $\varphi \in \mathcal{E}^1_m (\Omega,C)$, for some $C > 0$ large enough;

2) the function $\varphi$ is a critical point of the functional $\Phi_{G,\mu} $, hence it is a  solution to the Hessian equation i.e.

$$(dd^c \varphi)^m \wedge \beta^{n-m} = G(\cdot,\varphi) \, \mu,$$
 in the sense of measures on $\Omega$; 

3) Moreover if $\Omega$ is strictly $m$-pseudoconvex, $G$ has a polynomial growth of degree $m$ and  $\mu = g \beta^n$ white $g \in L^p(\Omega)$, where $(p,m)$ satisfy the conditions \eqref{eq:Holder}, then  $\varphi \in \mathcal C^{\alpha} (\bar\Omega)$ for some $\alpha \in ]0,1[$ and $\varphi$ is a solution to the Dirichlet problem  (\ref{eq:DirPro}).

4) If  $G(z,0) \equiv 0$ in $\Omega$ and $G$ satisfies $(H3)$, $\varphi$ is a non trivial solution i.e. $\varphi < 0$ in $\Omega$.
\end{theorem}
Here   $G$ has a polynomial growth of degree $m$ means that there exists a constant $M_0 > 0$ such that for $\mu$-a.e. $z \in \Omega$ and any $t <  0$, we have
$$
G(z,t) \leq M_0 \vert t \vert^m.
$$

We first prove the following  lemma.
\begin{lemma} \label{eq:lem:Fun}
Assume that $(\Omega,\mu, G)$ satisfies the assumptions $(H0)$, $(H1)$ and $(H2)$. Then the following properties hold :

$(1)$  the functional $ \Phi_{G,\mu} $  is well defined on $ \mathcal{E}^1_m (\Omega)$ and is lower semi-continuous on  each set
$$
\mathcal{E}^1_m (\Omega,C) := \{\phi \in \mathcal{E}^1_m (\Omega,C) \, ; \, 0 \leq E_m (\phi) \leq C\},
$$
 with $C > 0$; 

$(2)$ the functional $\Phi_{G,\mu}$  is coercive  on $\mathcal{E}_m^1 (\Omega)$ i.e. there exists constants $\epsilon_0 > 0$ and $C_0 > 0$ such that for any $\phi \in \mathcal{E}_m^1 (\Omega)$,
\begin{equation} \label{eq:Properness}
 \Phi_{G,\mu} (\phi)  \geq  \epsilon_0 \, E_m (\phi) - C_0.
\end{equation}

In particular $ \Phi_{G,\mu}$ is bounded from bellow.
\end{lemma}
\begin{proof}
1) We already know by Lemma \ref{lem:FirstVariation} that $E_m$ is well defined  and lower semi-continuous on $\mathcal{E}_m^1 (\Omega)$.
It remains to prove that the functional defined as follows 
$$
L_{H,\mu} (\phi) := \int_\Omega H(z,\phi(z)) d \mu(z)
$$ 
is  well defined on $\mathcal{E}_m^1 (\Omega)$ and  upper  semi-continuous on $ \mathcal{E}^1_m (\Omega,C)$. 
Let us first  prove that $L :=  L_{H,\mu} $ is well defined on  $\mathcal{E}_m^1 (\Omega)$.   
Indeed the condition $(H2)$ implies that there exists a constant $M > 0$ such that 
\begin{equation} \label{eq:Hinequality}
H (z,t) \leq M (- t)^{m +1},
\end{equation}
 for any $t < 0$ and $\mu$-a.e. $z \in \Omega$.
 
Since $\mathcal{E}_m^1 (\Omega) \subset L^{m+1} (\Omega,\mu)$, it follows that the functional $L_{H,\mu}$ is well defined on $\mathcal{E}^1_m (\Omega)$.

Now we prove that $ L_{H,\mu}$ is continuous on $ \mathcal{E}^1_m (\Omega,C)$ for the $L^1_{loc}(\Omega)$-topology. 
 Indeed let $(\phi_j)$ be a sequence of $ \mathcal{E}^1_m (\Omega,C)$ which converges to
 $\phi \in \mathcal{E}^1_m (\Omega,C)$ in the $L^1_{loc}(\Omega)$-topology.  By Theorem \ref{thm:continuity}, the sequence $(\phi_j)$ converges to $\phi$ in $L^{m+1}(\Omega,\mu)$. 
 We want to show that
  $$
  \lim_{j \to + \infty} L_{H,\mu} (\phi_j) = L_{H,\mu} (\phi).
  $$  
  Indeed, $(\phi_j)$ is a Cauchy sequence in $L^{m+1}(\Omega,\mu)$ and then it admits a subsequence $(\phi'_j)$ such that for any $j \in \N$
  $$
  \Vert\phi'_{j+1} - \phi'_j\Vert_{L^{m+1}(\Omega,\mu)} \leq 2^{-j}.
  $$
  Therefore $F := \vert \phi'_0\vert + \sum_{j=0}^{+\infty} \vert \phi'_{j+1} - \phi'_j \vert \in L^{m+1}(\Omega,\mu)$ and we clearly have $\vert \phi'_j\vert \leq F$ $\mu$-a.e. in $\Omega$, for any $j \in \N$.
  
  It follows from \eqref{eq:Hinequality} that   $H (z,\phi'_j(z)) \leq M ( -\phi'_j(z))^m \leq  M F(z)^{m+1}$ for $\mu$-a.e. $z \in \Omega$ and any $j\in \N$.
  Since we still have $\phi'_j \to \phi$ in $L^{m+1}(\Omega,\mu)$, it follows from the Lebesgue convergence theorem and the continuity of $H$ in $t$, that 
  $$
  \lim_{j \to + \infty} L_{H,\mu} (\phi'_j) = L_{H,\mu} (\phi).
  $$
  By the same reasoning as above we see that any subsequence of $(\phi_j)$ satisfies the same property which means that  the sequence $(L_{H,\mu} (\phi_j))$ has a unique limit point equal to $L_{H,\mu} (\phi).$
  This proves the required statement.

\smallskip
2) To prove coercivity, observe that the condition $(H2)$ implies that $0 < \theta_1 < \lambda_1^m \slash (m +1) =: \theta$ and  $H (z,t) \leq \theta_1 \vert t\vert ^{m +1}$ for $\mu$-a.e.  $z \in \Omega$ and any $ t  \leq  t_0$, for some   $t_0 < 0 $.

Observe that   $H$ is non increasing in the last variable since $\partial_t H = - G \leq 0$. Hence for any $\phi \in \mathcal{E}_m^1 (\Omega)$, we have

\begin{eqnarray*}
\Phi_{G,\mu} (\phi) & \geq & E_m (\phi) - \theta_1  \int_{\{\phi <  t_0\}} (-\phi)^{m + 1} d \mu -  \int_\Omega H (z,t_0) d \mu (z) \\
& \geq & E_m (\phi) - \theta_1  \int_\Omega  (-\phi)^{m + 1} d \mu - C_0,
\end{eqnarray*}
where $C_0 :=   \int_\Omega H (z,t_0) d \mu (z).$

On the other hand, by definition of $\lambda_1$, we have 
$$
E_m(\phi) \geq \lambda_1^m\slash (m+1) \int_\Omega (-\phi)^{m+1} d \mu = \theta \int_\Omega (-\phi)^{m+1} d \mu.
$$
Hence
\begin{eqnarray*}
 \Phi_{G,\mu} (\phi) & \geq &  E_m (\phi) - ( \theta_1 \slash \theta) E_m (\phi) - C_0 \\
 &  \geq&  (1- \theta_1 \slash \theta) E_m (\phi) - C_0,
\end{eqnarray*}
which proves the estimate (\ref{eq:Properness}),  since $\epsilon_0 :=  (1-\theta_1 \slash \theta) >0$. 
\end{proof}
Now we can prove Theorem \ref{thm:Variational2}.
\begin{proof} Observe that by \eqref{eq:Properness}, the functional $\Phi_{G,\mu} $ is bounded from below
and 
$$
\lim_{E_m(\varphi) \to + \infty}  \Phi_{G,\mu} (\phi) = + \infty.
$$
Therefore there exists a constant $C >1$ such that
$$
\inf \{ \Phi_{G,\mu} (\phi) \, ; \, \phi  \in \mathcal{E}_m^1 (\Omega)\} = \inf \{ \Phi_{G,\mu} (\phi) \, ; \, \phi  \in \mathcal{E}_m^1 (\Omega,C)\}\cdot
$$
Since the functional $\Phi_{G,\mu}$ is lower semi-continuous on the compact set $\mathcal{E}_m^1 (\Omega,C)$, it achives its minimum at some $\varphi \in \mathcal{E}_m^1 (\Omega,C)$.

As a result we can deduce as in the proof of Theorem \ref{thm:Variational1}  that  $\varphi$ is a critical point of $\Phi_{G,\mu}$ and then $\varphi$ is a solution to the equation (\ref{eq:DirPro}).

\smallskip

Let us prove that when $G(z,0) \equiv 0$ in $\Omega$, the condition $(H3)$ implies that  $\varphi \not \equiv 0$ in $\Omega$.  More precisely   we will prove that the condition $(H3)$ implies that the infimum of $\Phi_{G,\mu}$ is negative.
 
 Indeed,  let $u$ be a non trivial solution of the eigenvalue problem (\ref{eq:VP-Hessian}), which implies that  $E (u) = \theta  \int_\Omega (-u)^{m+1} d \mu$.

By $(H3)$ there exists  $ \theta_2  > 0$ such that $\liminf_{t \to 0^+} \frac{ H (z,t)}{t^{m + 1}} > \theta_2 > \theta$ for $\mu$-a.e. $z \in \Omega$.

 Therefore for any $t < 0$ we have
\begin{eqnarray*}
\Phi_{G,\mu} (tu) & = & (-t)^{m+1} E_m (u) - \int_\Omega H(z,t u(z)) d \mu (z) \\
&=& \theta  (-t)^{m+1}  \int_\Omega (-u)^{m+1} d \mu -  \int_\Omega  H (z,t u(z)) d \mu (z).
\end{eqnarray*}
Observe that for $t < 0$, we have
$$
\int_\Omega H (z,t u) d \mu =  \int_\Omega \theta   (t u)^{m+1} (H (z,t u) \slash \theta  ( t u)^{m+1})  d \mu (z),
$$
hence
$$
\Phi_{G,\mu} (tu) = \theta  (-t)^{m+1}  \int_\Omega (-u)^{m+1}  (1 - h (z,t)) d \mu (z),
$$
where 
$$
h (z,t) := (H (z,t u (z)) \slash \theta  (t u (z))^{m+1}) 
$$
is a uniformly bounded function on $\Omega$ by \eqref{eq:Hinequality} and satisfies  $\lim_{t \to 0^-} h(z,t) \geq \theta_2 \slash \theta > 1$  for $\mu$-a.e. $z \in \Omega$.

By the Fatou lemma we conclude that
$$
\limsup_{t \to 0^-}  \,  (-t)^{- m-1} \Phi_{G,\mu} (tu)  \leq (\theta-  \theta_2)  \int_\Omega (-u)^{m+1} d \mu <  0,
$$
since $\theta_2 > \theta$. This implies that for $t< 0$ small enough, $\Phi_{G,\mu} (t u) < 0$, hence $\Phi_{G,\mu} (\varphi) = \inf \Phi_{G,\mu} < 0$. Which proves that $\varphi \not \equiv 0$ in $\Omega$.

Now  prove that $\varphi $  is H\"older continuous and $\varphi \equiv 0$ in $\partial \Omega$ under the assumptions of the statement 3, we use the same argument as in the proof of Theorem \ref{thm:Variational1} based on Lemma \ref{lem:Integ}. 
\end{proof}

\subsection{Proof of Theorem 1.2 }
Here we will apply  Theorem \ref{thm:Variational2}  to prove Theorem 1.2.

\begin{proof}
We want apply  Theorem \ref{thm:Variational2} with a function $G (z,t) = f(z,t)^m$ for $(z,t) \in \Omega \times ]-\infty, 0]$. 

The hypothesis $(H1)$ is clearly satisfied. Let us prove that the hypothesis $(H2)$ is satisfied. Indeed,
since $\partial_t f(z,t) \geq - \lambda_0$, it follows that for $\mu$-a.e. $z \in \Omega$ and $t < 0$, 
$$
f (z,t) =  \int_0^t \partial_t f(z,s) d s \leq f(z,0) -  \lambda_0 t.
$$
Hence  for $\mu$-a.e. $z \in \Omega$ and $t < 0$,
$$
G (z,t) = f(z,t)^m \leq (M_0 -  \lambda_0 t)^m,
$$ 
where $M_0 := \Vert f(\cdot,0)\Vert_{L^\infty(\Omega)}$.

Therefore  for $\mu$-a.e. $z \in \Omega$ and $t < 0$,
$$
H(z,t) = \int_t^0 G(z,s) d s \leq  \frac{1}{\lambda_0(m+1)} (M_0 - \lambda_0 t)^{m +1}.
$$
The hypothesis $(H2)$ follows since $\lambda_0 < \lambda_1$.
Therefore the theorem follows from Theorem \ref{thm:Variational2}.
\end{proof}

\subsection{Applications}

We can deduce two interesting corollaries.

\begin{corollary} \label{coro:deformation1} Assume that $(\Omega,\mu)$ satisfies the hypothesis $(H0)$. Then for any number  $0 < \lambda < \lambda_1 = \lambda_1 (m,\Omega,\mu)$, the Hessian equation
\begin{equation}  \label{eq:deformation}
(dd^c u)^m \wedge \beta^{n-m} = (1 - \lambda u)^m \mu,
\end{equation}
admits a solution $u \in  \mathcal{E}^1_m (\Omega)$.

Moreover if $\Omega$ is strictly $m$-pseudoconvex and $\mu = g \beta^n$  with $g \in L^p(\Omega)$, where $(m,p )$ satisfy the conditions \eqref{eq:Holder},  then  $u \in \mathcal C^{\alpha} (\bar\Omega)$ for some $\alpha \in ]0,1[$ and   $u\mid_{\partial \Omega} \equiv 0$.
\end{corollary}
\begin{proof}  Apply the previous theorem with $G (z,t) = (1 -  \lambda t)^m$ for $(z \in \Omega$ and $t \leq 0$), and so 
$$
H (z,t) = (1 - \lambda t)^{m +1} \slash (\lambda (m + 1)), \, (z,t) \in \Omega \times ]-\infty,0].
$$ 
It is clear that the conditions $(H_1)$  and $(H2)$  are satisfied if $\lambda  < \lambda_1$.
\end{proof}

\begin{corollary} \label{coro:deformation2} Assume that $(\Omega,\mu)$ satisfies the hypothesis $(H0)$ and  let  $0 < k < m$, $a \geq 0$ and  $\lambda >0$ be fixed. Then the Hessian equation
$$
(dd^c u)^m \wedge \beta^{n-m} = (a - \lambda u)^k \mu,
$$
admits a non trivial solution $u \in  \mathcal{E}^1_m (\Omega)$.

Moreover if $\Omega$ is strictly $m$-pseudoconvex and $\mu = g \beta^n$ with $g \in L^p(\Omega)$, where $(m,p )$ satisfy the conditions \eqref{eq:Holder},  then  $u \in \mathcal C^{\alpha} (\bar\Omega)$ for some $\alpha \in ]0,1[$ and   $u\mid_{\partial \Omega} \equiv 0$.
\end{corollary}

\begin{proof}  Apply the previous theorem with $G (z,t) = (a-\lambda t)^k$ for $(z,t) \in \Omega \times ]-\infty,0]$, and so
$$
H(z,t) = (a - \lambda t)^{k +1} \slash (\lambda (k + 1)), (z,t) \in \Omega \times ]-\infty,0].
$$
 We see in the same way that the conditions $(H_1)$, $(H2)$ and $(H3)$ are satisfied since $0 < k < m$. 
\end{proof} 
 \smallskip
 \smallskip

{\bf Question 1 :} Is the eigenvalue  in Theorem  \ref{thm:Variational1} simple i.e. the eigenfunction is unique up to a positive multiplicative constant ? 

\smallskip
\smallskip
{\bf Question 2 :} In the general case when $m < n$ and $\mu = g \beta^n$, where $g \in L^p (\Omega)$ with $p > n\slash m$, is the eigenvalue  in Theorem  \ref{thm:Variational1} bounded, continuous or H\"older continuous in $\bar \Omega$ ? The same question can be asked for the solution given by Theorem 5.1 and its corollaries.

\smallskip
\smallskip

{\bf Question 3 :} Is the solution provided by Corollary \ref{coro:deformation1} (resp. Corollary \ref{coro:deformation2})   unique ?

When $\mu := g \beta^n$ with $0 < g \in C^{\infty} (\bar \Omega)$, it follows from \cite{BaZe23} that  the answer to the first and second question is positive when $m=n$. 
We believe that this is  still true in general. 

\smallskip
\smallskip

{\bf Aknowledgements :}  
The authors would like to thank the referee for his careful reading of the first version of this article and for his useful comments and suggestions which made it possible to correct certain statements and improve the presentation of the article.

 \end{document}